\documentclass[12pt,draft]{amsart}
\usepackage{setspace,units}

\newtheorem{theorem}{Theorem}[section]

\newtheorem{Corollary}[theorem]{Corollary}

\theoremstyle{definition}

\newtheorem{remark}[theorem]{Remark}

\numberwithin{equation}{section}

\RequirePackage{amsmath}
\RequirePackage{amssymb}
\RequirePackage{amsthm}
 \RequirePackage{epsfig}

\begin{document}

\date{}
\title[Duffin-Schaeffer conjecture]
{A note on the Duffin-Schaeffer conjecture}

\author{Liangpan Li}

\address{Department of Mathematical Sciences,
Loughborough University, LE11 3TU, UK}
 \email{mall4@lboro.ac.uk}



\date{}

\begin{abstract}
Given a sequence of real numbers $\{\psi(n)\}_{n\in\mathbb{N}}$ with
$0\leq \psi(n)<1$, let $W(\psi)$ denote the set of $x\in[0,1]$ for
which $|xn-m|<\psi(n)$ for infinitely many coprime pairs
$(n,m)\in\mathbb{N}\times\mathbb{Z}$. The purpose of this note is to
show that if there exists an $\epsilon>0$ such that
$\sum_{n\in\mathbb{N}}\psi(n)^{1+\epsilon}\cdot\frac{\varphi(n)}{n}=\infty,$
then the Lebesgue measure of $W(\psi)$ equals 1.
\end{abstract}

\maketitle


\section{Introduction to the Duffin-Schaeffer conjecture}

Given arithmetic functions $\psi:\mathbb{N}\rightarrow[0,\infty)$
and $\omega:\mathbb{N}\rightarrow\mathbb{N}$ with $\omega$
increasing, let $W(\psi,\omega)$ denote the set of $x\in[0,1]$ for
which $|x\cdot\omega(n)-m|<\psi(n)$ for infinitely many coprime
pairs $(\omega(n),m)\in\mathbb{N}\times\mathbb{Z}$. If $\omega$ is
the identity map then we simply write $W(\psi)$ for
$W(\psi,\omega)$. In this note we will study a long-standing
conjecture of Duffin and Schaeffer (\cite{DuffinSchaeffer}), who
claimed
\[\sum_{n\in\mathbb{N}}\psi(n)\cdot\frac{\varphi(\omega(n))}{\omega(n)}=\infty\Leftrightarrow{\mathcal M}(W(\psi,\omega))=1,\]
where $\varphi$ denotes Euler's totient function, $\mathcal M$ is
the Lebesgue measure on $\mathbb{R}$. Since the necessity part of
this conjecture follows from the first Borel-Cantelli lemma (see
e.g. \cite[Theorem 2.4]{Harman98}), we need only focus on the
sufficiency one.

There are many partial results towards the sufficiency part of the
Duffin-Schaeffer conjecture, for example, by assuming any of the
following additional conditions:
\begin{itemize}
\item A1: $\omega$ is the identity map, $n\mapsto n\psi(n)$ is non-increasing
(Khintchine \cite{Khintchine});
\item A2: $\omega$ is the identity map, $\psi(n)\leq \frac{c}{n}$ for some $c>0$
(Vaaler \cite{Vaaler});
\item A3: $\omega(m)$ and $\omega(n)$ are coprime for any $m\neq n$, $\psi$ is
arbitrary (Strauch \cite{Strauch});
\item A4: $\{\omega(n)\}$ is a lacunary sequence, $\psi$ is arbitrary
(Harman \cite{Harman});
\item A5: $\psi(n)\geq c(\frac{\varphi(\omega(n))}{\omega(n)})^R$ for
some $c,R>0$ (Harman \cite{Harman}).
\end{itemize}
We should also note the following beautiful breakthroughs:
\begin{itemize}
\item PV: The higher-dimensional Duffin-Schaeffer conjecture, also known as the Sprind\u{z}uk conjecture (\cite{Sprindzuk}), was solved
by Pollington and Vaughan (\cite{PV1,PV2});
\item BH: $\sum_{n\in\mathbb{N}}\psi(n)=\infty\Rightarrow \mbox{Hausdorff dimension of}\
W(\psi)=1$ (Baker-Harman \cite[Theorem 10.7]{Harman98}).
\end{itemize}
Recently,  Haynes, Pollington and Velani (\cite[Corollary
1]{Haynes}, see also \cite{Harman98})  established the
Duffin-Schaeffer conjecture
 by assuming the following extra divergence condition:
\begin{itemize}
\item HPV1:
$\sum_{n\in\mathbb{N}}\big(\frac{\psi(n)}{n}\big)^{1+\epsilon}\cdot\varphi(n)=\infty$.
\end{itemize}
Then as an immediate application of the HPV1 theorem and the
Beresnevich-Velani Transference Principle (\cite[Theorem 2]{BV06}),
they  also studied the Hausdorff dimensional Duffin-Schaeffer
conjecture (\cite[Conjecture 2]{BV06}) and were able to improve the
BH theorem to (\cite[Theorem 2]{Haynes}):
\begin{itemize}
\item HPV2:
$\sum_{n\in\mathbb{N}}\big(\frac{\psi(n)}{n}\big)^{1-\epsilon}\cdot\varphi(n)=\infty\
(\forall\epsilon>0)\Rightarrow \mbox{Hausdorff dim. of}\ W(\psi)=1$.
\end{itemize}

The purpose of this note is to adapt some ideas of Harman in
\cite{Harman} to improve the HPV1 theorem as follows:
\begin{theorem}\label{maintheorem} Let $\psi:\mathbb{N}\rightarrow[0,1)$ be any function.
Then ${\mathcal M}(W(\psi))=1$ if there exists an $\epsilon>0$ such
that
\[\sum_{n\in\mathbb{N}}\psi(n)^{1+\epsilon}\cdot\frac{\varphi(n)}{n}=\infty.\]
\end{theorem}

Before proceeding to the proof of Theorem \ref{maintheorem} in the
next section, let us first give some remarks and corollaries.

\begin{remark} We point out by constructing an example that the function $\psi$  assumed in Theorem \ref{maintheorem} cannot
be extended to all non-negative arithmetic functions. Since
$\{\frac{\varphi(n)}{n}:n\in\mathbb{N}\}$ is dense in $(0,1)$
(\cite{Schinzel}), there exists a sequence of positive integers
$n_2<n_3<n_4<\cdots$ so that
\[\frac{\varphi(n_k)}{n_k}\sim\frac{1}{k\cdot(\log k)^{2+\frac{\epsilon}{2}}},\]
where $\epsilon>0$ is any prescribed real number. We then define
$\psi(n_k)=\log k$ and
note\begin{align*}\sum_k\psi(n_k)^{1+\epsilon}\cdot\frac{\varphi(n_k)}{n_k}&\sim\sum_k\frac{1}{k\cdot(\log
k)^{1-\frac{\epsilon}{2}}}=\infty,\\
\sum_k\psi(n_k)\cdot\frac{\varphi(n_k)}{n_k}&\sim\sum_k\frac{1}{k\cdot(\log
k)^{1+\frac{\epsilon}{2}}}<\infty.\end{align*} Applying not Theorem
\ref{maintheorem} but the first Borel-Cantelli lemma gives
${\mathcal M}(W(\psi))=0$.
\end{remark}

\begin{Corollary}[HPV1] \label{corollary 13} Let $\psi:\mathbb{N}\rightarrow[0,\infty)$ be
any function. Then ${\mathcal M}(W(\psi))=1$ if there exists an
$\epsilon>0$ such that
\[\sum_{n\in\mathbb{N}}\big(\frac{\psi(n)}{n}\big)^{1+\epsilon}\cdot\varphi(n)=\infty.\]
\end{Corollary}

\begin{proof}
We divide $\mathbb{N}$ into two parts $\mathbb{X},\mathbb{Y}$, that
is, $n\in\mathbb{X}$ or $n\in\mathbb{Y}$ according as $\psi(n)<1$ or
not, and have two cases to consider.

\textsc{Case 1}: Suppose
$\sum_{n\in\mathbb{X}}\big(\frac{\psi(n)}{n}\big)^{1+\epsilon}\cdot\varphi(n)=\infty$.
Obviously,
$\sum_{n\in\mathbb{X}}\psi(n)^{1+\epsilon}\cdot\frac{\varphi(n)}{n}=\infty$.
Thus ${\mathcal M}(W(\psi))=1$ follows from applying Theorem
\ref{maintheorem}.

\textsc{Case 2}: Suppose
$\sum_{n\in\mathbb{Y}}\big(\frac{\psi(n)}{n}\big)^{1+\epsilon}\cdot\varphi(n)=\infty$.
In this case it is easy to prove that
$\sum_{n\in\mathbb{Y}}\frac{\psi(n)}{n}\cdot\varphi(n)=\infty$.
 Applying Harman's Condition A5 gives ${\mathcal M}(W(\psi))=1$.

This finishes the proof.
\end{proof}

\begin{remark}\label{remark14}
Similar to the proof of the HPV2 theorem, one may expect that as an
application of Theorem \ref{maintheorem} and the Beresnevich-Velani
Transference Principle the following proposition might be
true:\begin{itemize}
\item
$\sum_{n\in\mathbb{N}}\frac{\psi(n)^{1-\epsilon}}{n^{1-2\epsilon}}\cdot\varphi(n)=\infty\
(\forall\epsilon>0)\Rightarrow \mbox{Hausdorff dimension of}\
W(\psi)=1$.
\end{itemize}
We remark that even if the above proposition is true,  it cannot
give any genuine improvement of the HPV2 theorem. The reason is as
follows: If there are infinitely many $n\in\mathbb{N}$ so that
$\psi(n)\geq1$, then by the BH theorem  $W(\psi)$ has full Hausdorff
dimension. Thus to study what sufficient condition can guarantee
$W(\psi)$ has full Hausdorff dimension, we may assume without loss
of generality that $\psi(n)<1$ for all $n\in\mathbb{N}$. Suppose
this is the case, then
$\sum_{n\in\mathbb{N}}\frac{\psi(n)^{1-\epsilon}}{n^{1-2\epsilon}}\cdot\varphi(n)=\infty$
implies
$\sum_{n\in\mathbb{N}}\frac{\psi(n)^{1-2\epsilon}}{n^{1-2\epsilon}}\cdot\varphi(n)=\infty$.
Now we can apply the HPV2 theorem to ensure  $W(\psi)$ has full
Hausdorff dimension.
\end{remark}

\begin{Corollary}\label{corollary 15}
Let $\psi,\gamma:\mathbb{N}\rightarrow[0,\infty)$ be any functions
 with
$\limsup_{n\rightarrow\infty}\frac{\log \gamma(n)}{\log n}\leq0$.
Then $W(\psi)$ has full Hausdorff dimension if
\[\sum_{n\in\mathbb{N}}\psi(n)\cdot\gamma(n)=\infty.\]
\end{Corollary}

\begin{proof}
As discussed in  Remark \ref{remark14}, to prove Corollary
\ref{corollary 15} we may assume without loss of generality that
$\psi(n)<1$ for all $n\in\mathbb{N}$. Since
$\limsup_{n\rightarrow\infty}\frac{\log \gamma(n)}{\log n}\leq0$, we
observe that for any $\epsilon>0$,
\[\psi(n)\cdot\gamma(n)\leq\big(\frac{\psi(n)}{n}\big)^{1-\epsilon}\cdot\varphi(n)\]
holds for sufficiently large $n$. According to the HPV2 theorem we
obtain that $W(\psi)$ has full Hausdorff dimension. This finishes
the proof.
\end{proof}

\section{Proof of the main result}

\textsc{A theorem of Harman}: Let $\psi_1,\ldots,\psi_k$ ($k\geq2$)
be functions of $n\in\mathbb{N}$, taking values in $[0,1]$. Write
\begin{equation}\theta(n)=\prod_{j=1}^k\psi_j(n),\end{equation}
and suppose for some positive reals $\delta$ and $K$, that for each
$n$ with $\theta(n)\neq0$,
\begin{equation}\max_{1\leq j\leq k}\frac{\theta(n)}{\psi_j(n)}\leq K\cdot\theta(n)^{\delta}.\end{equation} A theorem of Harman (\cite[Theorem 2]{Harman}) on the
higher-dimensional Duffin-Schaeffer conjecture claims that if
\begin{equation}\sum_{n\in\mathbb{N}}\theta(n)\cdot(\frac{\varphi(n)}{n})^k=\infty,\end{equation}
then ${\mathcal M}_k(W(\psi_1,\ldots,\psi_k))=1$, where  ${\mathcal
M}_k$ is the $k$-dimensional Lebesgue measure,
$W(\psi_1,\ldots,\psi_k)$ is the set of $(x_1,\ldots,x_k)\in[0,1]^k$
for which
\begin{equation}|nx_j-m_j|<\psi_j(n)\ \ \ (n,m_j)=1,\ \ \ j=1,\ldots,k\end{equation}
for infinitely many
$(n,m_1,\ldots,m_k)\in\mathbb{N}\times\mathbb{Z}^k$.

\textsc{Proof of Theorem \ref{maintheorem}}:  Obviously, we may
assume that $0<\epsilon<1$. We divide $\mathbb{N}$ into two parts
$\mathbb{X},\mathbb{Y}$, that is,
$n\in\mathbb{X}\Leftrightarrow\psi(n)\leq\big(\frac{\varphi(n)}{n}\big)^{2/\epsilon},$
$n\in\mathbb{Y}\Leftrightarrow\psi(n)>\big(\frac{\varphi(n)}{n}\big)^{2/\epsilon}.$
Since $\psi(n)<1$ for each $n\in\mathbb{N}$, we can deduce from
$\sum_{n\in\mathbb{N}}\psi(n)^{1+\epsilon}\cdot\frac{\varphi(n)}{n}=\infty$
that $\sum_{n\in\mathbb{N}}\psi(n)\cdot\frac{\varphi(n)}{n}=\infty.$
Now we have two cases to consider.

\textsc{Case 1}: Suppose
$\sum_{n\in\mathbb{Y}}\psi(n)\cdot\frac{\varphi(n)}{n}=\infty$. Then
according to Harman's Condition A5, we have ${\mathcal
M}(W(\psi))=1$.

\textsc{Case 2}: Suppose
$\sum_{n\in\mathbb{Y}}\psi(n)\cdot\frac{\varphi(n)}{n}<\infty$.
Obviously,
$\sum_{n\in\mathbb{Y}}\psi(n)^{1+\epsilon}\cdot\frac{\varphi(n)}{n}<\infty$.
We define $\Psi(n)=\psi(n)$ if $n\in\mathbb{X}$, and $\Psi(n)=0$ if
$n\in\mathbb{Y}$. It is easy to check that for any $n\in\mathbb{N}$,
\[\Psi(n)\leq\big(\frac{\varphi(n)}{n}\big)^{2/\epsilon},\]
and
\[\sum_{n\in\mathbb{N}}\Psi(n)^{1+\epsilon}\cdot\frac{\varphi(n)}{n}=\infty.\]
Now it is time  to apply Harman's theorem introduced in the
beginning of this section. Let $k=2$ and let $\psi_1(n)=\Psi(n)$,
$\psi_2(n)=\Psi(n)^{\epsilon}\cdot\frac{n}{\varphi(n)}$. Then for
any $n\in\mathbb{N}$ with $\theta(n)\neq0$, equivalently,
$\Psi(n)\neq0$, one has
\begin{equation}
\frac{\theta(n)}{\psi_1(n)}=\Psi(n)^{\epsilon}\cdot\frac{n}{\varphi(n)}\leq\Psi(n)^{\epsilon/2}
=\big(\Psi(n)^{1+\epsilon}\big)^{\frac{\epsilon}{2(1+\epsilon)}}\leq\theta(n)^{\frac{\epsilon}{2(1+\epsilon)}},
\end{equation}
and similarly,
\begin{equation}
\frac{\theta(n)}{\psi_2(n)}=\Psi(n)=\big(\Psi(n)^{1+\epsilon}\big)^{\frac{1}{1+\epsilon}}\leq\theta(n)^{\frac{1}{1+\epsilon}}.
\end{equation}
Noting $\epsilon\in(0,1)$, and
$\theta(n)=\Psi(n)^{1+\epsilon}\cdot\frac{n}{\varphi(n)}\leq\Psi(n)^{1+\frac{\epsilon}{2}}<1$,
we have
\begin{equation}\max_{1\leq j\leq2}\frac{\theta(n)}{\psi_j(n)}\leq\theta(n)^{\frac{\epsilon}{2(1+\epsilon)}}.\end{equation}
Observe also that $\psi_1(n)<1$,
$\psi_2(n)=\Psi(n)^{\epsilon}\cdot\frac{n}{\varphi(n)}\leq\Psi(n)^{\frac{\epsilon}{2}}<1$,
and
\begin{equation}\sum_{n\in\mathbb{N}}\theta(n)\cdot\big(\frac{\varphi(n)}{n}\big)^2
=\sum_{n\in\mathbb{N}}\Psi(n)^{1+\epsilon}\cdot\frac{\varphi(n)}{n}=\infty.\end{equation}
With these preparations we can apply Harman's theorem to get
${\mathcal M}_2(W(\psi_1,\psi_2))=1$. Now we note an elementary
relation, that is, $W(\psi_1,\psi_2)\subset W(\psi_1)\times
W(\psi_2).$ As an immediate consequence, ${\mathcal
M}(W(\psi_1))=1$. Since $\psi_1\leq\psi$, we have $W(\psi_1)\subset
W(\psi)$, which gives ${\mathcal M}(W(\psi))=1$. This finishes the
 proof of Theorem \ref{maintheorem}.

\textbf{Acknowledgements}. The author would like to thank the
anonymous referee for pointing out one mistake made in the earlier
version and careful reading of the whole manuscript. He also thanks
Hao Wang for helpful discussions. This work was partially supported
by the NSF of China (11001174).


\begin{thebibliography}{9}




\bibitem{BV06}
V. Beresnevich, S. Velani, \emph{A mass transference principle and
the Duffin-Schaeffer conjecture for Hausdorff measures}, Ann. Math.
\textbf{164} (2006), 971--992.



\bibitem{DuffinSchaeffer}
R.~J. Duffin, A.~C. Schaeffer,\emph{ Khintchine's problem in metric
Diophantine approximation}, Duke Math. J. \textbf{8} (1941),
243--255.



\bibitem{Harman}
G. Harman, \emph{Some cases of the Duffin and Schaeffer conjecture},
Quart. J. Math. Oxford \textbf{41} (1990), 395--404.

\bibitem{Harman98}
G. Harman, Metric Number Theory, Clarendon Press, Oxford, 1998.


\bibitem{Haynes}
A.~K. Haynes, A.~D. Pollington, S.~L. Velani, \emph{The
Duffin-Schaeffer Conjecture with extra divergence}, Math. Ann.
\textbf{353} (2012), 259--273.

\bibitem{Khintchine}
A. Khintchine, \emph{Einige S\"{a}tze \"{u}ber Kettenbr\"{u}che mit
Anwendungen auf die Theorie der Diophantischen Approximationen},
Math. Ann. \textbf{92} (1924), 115--125.






\bibitem{PV1}
A.~D. Pollington, R.~C. Vaughan, \emph{The $k$-dimensional Duffin
and Schaeffer conjecture}, Journal de Th\'{e}orie des Nombres de
Bordeaux \textbf{1} (1989), 81--88.

\bibitem{PV2}
A.~D. Pollington, R.~C. Vaughan, \emph{The $k$-dimensional Duffin
and Schaeffer conjecture}, Mathematika \textbf{37} (1990), 190--200.

\bibitem{Schinzel}
A. Schinzel, W. Sierpi\'{n}ski, \emph{Sur quelques
propri\'{e}t\'{e}s des fonctions $\varphi(n)$ et $\sigma(n)$}, Bull.
Acad. Polon. Sci. Cl. III. \textbf{2} (1954), 463--466.



\bibitem{Sprindzuk}
V. Sprind$\breve{z}$uk, Metric Theory of Diophantine Approximation,
John Wiley \& Sons, New York, 1979. (English translation)

\bibitem{Strauch}
O. Strauch, \emph{Some new criterions for sequences which satisfy
Duffin-Schaeffer conjecture, I}, Acta Math. Univ. Comenian.
\textbf{42}-\textbf{43} (1983), 87--95.


\bibitem{Vaaler}
J.~D. Vaaler, \emph{On the metric theory of Diophantine
approximation}, Pacific J. Math. \textbf{76} (1978), 527--539.


\end{thebibliography}
\end{document}